\documentclass[twoside, 11pt]{article}
\usepackage{amssymb, amsmath, mathrsfs, amsthm}%, wasysym}
\usepackage{graphicx}
\usepackage{color}
\usepackage{float}

\newcommand{\bd}{\begin{description}}
\newcommand{\ed}{\end{description}}
\newcommand{\bi}{\begin{itemize}}
\newcommand{\ei}{\end{itemize}}
\newcommand{\be}{\begin{enumerate}}
\newcommand{\ee}{\end{enumerate}}
\newcommand{\beq}{\begin{equation}}
\newcommand{\eeq}{\end{equation}}
\newcommand{\beqs}{\begin{eqnarray*}}
\newcommand{\eeqs}{\end{eqnarray*}}

\newcommand{\rn}[1]{{\color{red} \bf #1}}

\definecolor{DarkGreen}{rgb}{0.2, 0.6, 0.3}

\newtheorem{theorem}{Theorem}
\newtheorem{conjecture}{Conjecture}

\newtheorem{lemma}{Lemma}

\newtheorem{corollary}[theorem]{Corollary}

\newtheorem{fact}{Fact}

\newtheorem{problem}{Problem}

\begin{document}
\title{\bf All partitions have small parts - Gallai-Ramsey numbers of bipartite graphs}

\author{Haibo Wu\footnote{School of Science, Northwestern Polytechnical University, Xi'an, P. R. China. \tt{whb@nwpu.edu.cn}}, Colton Magnant\footnote{Department of Mathematical Sciences, Georgia Southern University, Statesboro, GA 30460 USA. \tt{cmagnant@georgiasouthern.edu}}, Pouria Salehi Nowbandegani\footnote{Department of Mathematics, Vanderbilt University, Nashville, TN 37212 USA \ttfamily{pouria.salehi.nowbandegani@vanderbilt.edu}}, Suman Xia\footnote{School of Mathematical Sciences, LMAM, Peking University, Beijing 100871, P. R. China \ttfamily{xiasuman@pku.edu.cn}}}

% ----------------------------------------------------
\maketitle
% ----------------------------------------------------

\begin{abstract}
Gallai-colorings are edge-colored complete graphs in which there are no rainbow triangles. Within such colored complete graphs, we consider Ramsey-type questions, looking for specified monochromatic graphs. In this work, we consider monochromatic bipartite graphs since the numbers are known to grow more slowly than for non-bipartite graphs. The main result shows that it suffices to consider only $3$-colorings which have a special partition of the vertices. Using this tool, we find several sharp numbers and conjecture the sharp value for all bipartite graphs. In particular, we determine the Gallai-Ramsey numbers for all bipartite graphs with two vertices in one part and initiate the study of linear forests.
\end{abstract}

% --------------------
\section{Introduction}
% --------------------

Ramsey numbers have been a hot topic in mathematics for decades now due to their intrinsic beauty, wide applicability, and overwhelming difficulty despite somewhat misleadingly simple statements. The notion of ``order amid chaos'' defines the entire concept while applications to many different areas of mathematics and other sciences drive the study of new directions and extensions. See \cite{MR1670625} for a dynamic survey of known small Ramsey numbers and \cite{R04} for a dynamic survey of applications of Ramsey Theory.

Recall that the \emph{Ramsey number} $R(p,q)$ is the minimum integer $n$ such that, in every coloring of the edges of the complete graph on $n$ vertices using red and blue, there is either a red clique of order $p$, or a blue clique of order $q$. For more general graphs $G$ and $H$, let $R(G, H)$ denote the minimum integer $n$ such that, in every coloring of the edges of the complete graph on $n$ vertices using red and blue, there is either a red copy of $G$ or a blue copy of $H$.% We are now able to state their conjecture.

We consider edge-colorings of complete graphs which contain no rainbow triangles. This restricted class of colorings is particularly interesting due to the powerful structure provided by the following result of Gallai.

\begin{theorem}[\cite{CE97, G67, GS04}]\label{Thm:G-Part}
In any edge-coloring of a complete graph with no rainbow triangle, there exists a partition of the vertices into at least two parts (called a \emph{Gallai partition} or \emph{G-partition} for short) such that, there are at most two colors on the edges between the parts, and only one color on the edges between each pair of parts.
\end{theorem}

In light of this result, we say that a colored complete graph with no rainbow triangle is a \emph{Gallai coloring} (or \emph{G-coloring} for short). Closely related to results in \cite{FGP15}, Fox et al.~posed a conjecture about monochromatic complete graphs. In order to concisely state their conjecture, we must present some definitions. 

Given a graph $H$, the (\emph{$k$-colored}) \emph{Gallai-Ramsey number} $gr_{k}(K_{3} : H)$ is defined to be the minimum integer $n$ such that every $k$-coloring (using all $k$ colors) of the complete graph on $n$ vertices contains either a rainbow triangle or a monochromatic copy of $H$.

We refer to the survey of rainbow generalizations of Ramsey Theory \cite{FMO10, FMO14} for more information on this topic and a complete list of known results involving Gallai-Ramsey numbers. We are now able to state the conjecture of Fox et al.

\begin{conjecture}[\cite{FGP15}]\label{Conj:Fox}
For $k\ge 1$ and $p \geq 3$,
$$
gr_{k}(K_{3} : K_{p}) = \begin{cases}
(R(p, p) - 1)^{k/2} + 1 & \text{ if $k$ is even,}\\
(p - 1)(R(p, p) - 1)^{(k - 1)/2} + 1 & \text{ if $k$ is odd.}
\end{cases}
$$
\end{conjecture}

The case where $p = 3$ was actually verified first in 1983 by Chung and Graham \cite{CG83} and then again over the years in different contexts.

\begin{theorem}[\cite{AI08,CG83,GSSS10}]\label{Thm:grK3}
For $k \geq 1$,
$$
gr_{k}(K_{3} : K_{3}) = \begin{cases}
5^{k/2} + 1 & \text{if $k$ is even,}\\
2\cdot 5^{(k-1)/2} + 1 & \text{if $k$ is odd.}
\end{cases}
$$
\end{theorem}

The next case of this conjecture was recently proven in \cite{LMSSS17}.

\begin{theorem}[\cite{LMSSS17}]\label{Thm:Main}
For $k\ge 1$,
$$
gr_{k}(K_{3} : K_{4}) = \begin{cases} 17^{k/2} + 1 & \text{ if } k \text{ is even,}\\
3\cdot 17^{(k - 1)/2} + 1 & \text{ if } k \text{ is odd.}
\end{cases}
$$
\end{theorem}

The landscape for finding monochromatic bipartite graphs is very different. Where non-bipartite monochromatic graphs yield exponential functions of the number of colors $k$, as seen in the previous results, bipartite monochromatic graphs yield linear functions of $k$ (see Theorem~\ref{Thm:Dichotomy} below). In this work, we therefore consider the Gallai-Ramsey question for bipartite graphs, particularly complete bipartite graphs. 
Given a bipartite graph $H$, let $s(H)$ be the order of the smallest part of the bipartition of $H$.
Our main tool, which allows us to prove several results for different classes of bipartite graphs, is the following reduction theorem.

\begin{theorem}\label{Lem:OnlySmall}
Given a bipartite graph $H$ and a positive integer $R \geq R(H, H)$, if every G-coloring of $K_{R}$ using only $3$ colors, in which all parts of a G-partition have order at most $s(H) - 1$, contains a monochromatic copy of $H$, then 
$$
gr_{k}(K_{3} : H) \leq R + (s(H) - 1)(k - 2).
$$
\end{theorem}

Essentially, this result says that if we intend to prove that
$$
gr_{k}(K_{3} : H) \leq R + (s(H) - 1)(k - 2),
$$
then it suffices to prove that every G-coloring of $K_{R}$ using only $3$ colors, in which all parts of a G-partition are small, contains a monochromatic copy of $H$.

%\begin{theorem}\label{Lem:OnlySmall}
%Given a bipartite graph $H$ and a positive integer $R \geq R(H, H)$, if we intend to prove that
%$$
%gr_{k}(K_{3} : H) \leq R + (s(H) - 1)(k - 2),
%$$
%then it suffices to prove that every G-coloring of $K_{R}$ using only $3$ colors, in which all parts of a G-partition have order at most $s(H) - 1$, contains a monochromatic copy of $H$.
%\end{theorem}

The remainder of this paper is organized as follows. Section~\ref{Sec:Prelim} contains several known useful results along with some helpful lemmas which will be used later. Section~\ref{Sec:Thm4} contains the proof of Theorem~\ref{Lem:OnlySmall}. The remaining sections use Theorem~\ref{Lem:OnlySmall} to produce Gallai-Ramsey results for certain classes of bipartite graphs except Section~\ref{Sec:Conj}, which contains our broad conjecture of the value for all bipartite graphs.

%%%%%%%%%%%%%%%%%%%%%%%
\section{Preliminaries}\label{Sec:Prelim}
%%%%%%%%%%%%%%%%%%%%%%%

We first state some known classical $2$-color Ramsey numbers for complete bipartite graphs.

\begin{theorem}[\cite{B83}] \label{Thm:RK23}
$R(K_{2,3},K_{2,3})=10$.
\end{theorem}

\begin{theorem}[\cite{HM91}] \label{Thm:RK33}
$R(K_{3, 3}, K_{3, 3}) = 18$.
\end{theorem}

\begin{theorem}[\cite{EHM91}] \label{Thm:ExooK2m}
If $4n-4 = 2^tk_1 \dots k_s$ with $t\geq 0$, and either $k_i=p_{i}^{r_i}+1$ where $p_i^{r_i}\equiv 3$ (mod 4) is a prime, or $k_i = 2(q_i^{u_i}+1)$ where $q_i^{u_i} \equiv 1$ (mod 4) is a prime power, then $R(K_{2,n},K_{2,n})\geq 4n-3$ if $t>0$, and $R(K_{2,n},K_{2,n})\geq 4n-4$ otherwise.
\end{theorem}

We will also use the following result concerning monochromatic stars.

\begin{theorem}[\cite{GS04}] \label{Thm:GS04}
Any Gallai coloring of $K_n$ contains a monochromatic star $S_t$ with $t\geq \frac{2n}{5}$.
\end{theorem}

Our first lemma provides the lower bound for our results on the Gallai-Ramsey number for bipartite graphs. Indeed, we believe this lower bound to be sharp for all connected bipartite graphs (see Conjecture~\ref{Conj:GR-Bip}).

\begin{lemma}\label{Lem:BipLowBnd}
For a given connected bipartite graph $H$ with Ramsey number $R(H, H) = R$, we have
$$
gr_{k}(K_{3} : H) \geq R + (s(H) - 1)(k - 2).
$$
\end{lemma}

\begin{proof}
With $R = R(H, H)$, there exists a $2$-colored complete graph $G_{2}$ on $R - 1$ vertices containing no monochromatic copy of $H$. Given a colored complete graph $G_{i - 1}$, for each additional color $i$ with $3 \leq i \leq k$, we add $s(H) - 1$ vertices with all incident edges colored in color $i$ to create the graph $G_{i}$. The resulting graph $G_{k}$ is a $k$-colored complete graph on $R - 1 + (s(H) - 1)(k - 2)$ vertices containing no monochromatic copy of $H$ and no rainbow triangle.
\end{proof}

Note that this construction, and therefore this lower bound, does not suffice if the monochromatic graph $H$ is not bipartite. Indeed, it is easy to show that if $H$ is not bipartite, then $gr_{k}(K_{3} : H)$ grows as an exponential function of $k$, stated more precisely in the following result.

\begin{theorem}[\cite{GSSS10}]\label{Thm:Dichotomy}
Let $H$ be a fixed graph with no isolated vertices.  If $H$ is not bipartite, then $gr_{k}(K_{3} : H)$ is exponential in $k$.   If $H$ is bipartite, then $gr_{k}(K_{3} : H)$ is linear in $k$.
\end{theorem}

The following technical lemma will be used to eliminate many cases from our main results. It then suffices to consider only two possibilities: when a largest part of a G-partition is rather small or when this largest part is rather large.

\begin{lemma}\label{Lem:Middle}
Given positive integers $\ell, m, n$ where $\ell \leq m$ and $n \geq 3m - 2$, let $H = K_{\ell, m}$ be a complete bipartite graph and let $G$ be a G-coloring of $K_{n}$ with no monochromatic copy of $H$. In any G-partition of $G$, any largest part of the partition has order either at most $\ell - 1$ or at least $n - 2\ell + 2$.
\end{lemma}

\begin{proof}
Let $H$ and $G$ be as given and let $H_{1}$ be a largest part of a G-partition of $G$. Suppose, for a contradiction, that $\ell \leq |H_{1}| \leq n - 2\ell + 1$.

First suppose that $|H_{1}| \leq n - 2m + 1$. Since $H_{1}$ is a part of a G-partition, every other vertex of $G$ has one of only two colors on all edges to $H_{1}$. By the pigeonhole principle, there must be at least $m$ vertices outside $H_{1}$, say a set $S$, with all one color on edges to $H_{1}$. Since $|H_{1}| \geq \ell$, $S \cup H$ contains a monochromatic copy of $K_{\ell, m}$, a contradiction.

Finally suppose that $n - 2m + 2 \leq |H_{1}| \leq n - 2\ell + 1$. This means that there are at least $2\ell - 1$ vertices in $G \setminus H_{1}$. Since $n \geq 3m - 2$, we get $|H_{1}| \geq n - 2m + 2 \geq m$. By the pigeonhole principle, there is a set $S$ of at least $\ell$ vertices in $G \setminus H_{1}$ with the property that all edges between $S$ and $H_{1}$ have the same color. Then $S \cup H_{1}$ induces a monochromatic graph containing $K_{\ell, m}$, a contradiction completing the proof of Lemma~\ref{Lem:Middle}.
\end{proof}

%Even more specifically, it turns out that we may focus only on the case where a largest part of the G-partition is rather small.

\section{Proof of Theorem~\ref{Lem:OnlySmall}}\label{Sec:Thm4}

Recall that Theorem~\ref{Lem:OnlySmall} stated as follows. Given a bipartite graph $H$ and a given positive integer $R \geq R(H, H)$, if we intend to prove that
$$
gr_{k}(K_{3} : H) \leq R + (s(H) - 1)(k - 2),
$$
then it suffices to prove that every G-coloring of $K_{R}$ using only $3$ colors, in which all parts of the partition have order at most $s(H) - 1$, contains a monochromatic copy of $H$.

\begin{proof}
Let $H$ be the given bipartite graph with $a = s(H)$ and $b = |H| - a$. Let $G$ be a G-coloring of $K_{n}$ where $n = R + (a - 1)(k - 2)$. Define a set $T$ of vertices $\{v_{1}, v_{2}, \dots, v_{|T|}\}$ to be a largest set with the property that each vertex $v_{i}$ has all except possibly at most $(a - 1)$ of its edges to $G \setminus \{v_{1}, v_{2}, \dots, v_{i}\}$ in a single color, with the extra restriction that if any of these edges has a different color, it must be an edge of the form $v_{i}v_{i + \ell}$ where $\ell \leq a - 1$. Note that $|T| \leq (a - 1)k$ since if $|T| \geq (a - 1)k + 1$, then there would exist a color, say $i$, such that at least $a$ vertices among the first $(a - 1)k + 1$ vertices of $T$ have all edges to the rest of the graph in color $i$. This yields a monochromatic $K_{a, b}$ in color $i$, which contains the desired bipartite graph $H$.

By Lemma~\ref{Lem:Middle}, the largest part of any G-partition of $G$ has order either at least $n - 2a + 2$ or at most $a - 1$. For the proof of this lemma, it suffices to suppose the largest part $H_{1}$ has order at least $n - 2a + 2$. Then the vertices of $G \setminus H_{1}$ can be added to $T$, contradicting the maximality of $|T|$.

By Theorem~\ref{Thm:G-Part}, there is a G-partition of $G \setminus T$, say using colors red and blue. We now replace any vertices of $T$ (if they exist) that have red or blue edges to $G \setminus T$, replace all edges within the parts of the G-partition with a third color, and let $G'$ be the resulting $3$-colored complete graph. Note that $|G'| \geq R$. Certainly there is a G-partition of $G'$ using colors red and blue with each part having order at most $a - 1$. Also note that a rainbow triangle and a monochromatic copy of $H$ can easily be avoided within $T$ so it suffices to consider the $3$-coloring $G'$ of order $R$.
\end{proof}

%\begin{definition}
%Let $G,H$ be graphs with $|G|=n$. Then the extremal number of $G$, denoted ex($n,H$), is the maximum number of edges that $G$ can have such that it does not contain $H$ as a subgraph.
%\end{definition}

%\begin{theorem}[Z. F{\"u}redi (1994)] 
%link to the article for reference: http://www.sciencedirect.com/science/article/pii/S0097316596900679
%For any fixed $t\geq 1$, ex($n,K_{2,t+1}) = 1/2\sqrt{t}n^{3/2}+O(n^{4/3})$.
%\end{theorem}

%This result comes with a corresponding lower bound example on $n = (q^{2} - 1)/t$ vertices where $q$ is a prime, and $(1/2)\sqrt{t} n^{3/2} - (n/2)$ edges but no $K_{2, t + 1}$.

\section{Bipartite Graphs With $s(H)=2$}

Recall that $s(H)$ is the order of the smallest part of the bipartition of $H$. For general bipartite graphs $H$ with $s(H) = 2$, we obtain the following very broad result.

\begin{theorem}\label{Thm:sis2}
Let $H$ be a bipartite graph with $s(H) = 2$ and $R(H, H) = R$. Then for any integer $k \geq 2$, we have
$$
gr_{k}(K_{3} : H) = R + (k - 2).
$$
\end{theorem}

\begin{proof}
The lower bound follows from Lemma~\ref{Lem:BipLowBnd}.
%For the lower bound, let $K$ be a coloring of $K_{R_{m} - 1}$ that provides the sharpness for the Ramsey number $R(K_{2, m}, K_{2, m})$. Then $L(K, 1, k - 2)$ is a Gallai $k$-coloring of $K_{R_{2, m} + k - 3}$ with no monochromatic $K_{2, m}$.
For the upper bound, let $G$ be a G-coloring of $K_{n}$ where $n = R+k-2$ containing no monochromatic copy of $K_{2, m}$. % and consider a G-partition of $G$. Let $H_1$ be a largest part of the partition.
By Theorem~\ref{Lem:OnlySmall}, it suffices to consider a $3$-coloring $G'$ of $K_{R}$ with a G-partition in which all parts have order $1$. Since there are only two colors in the G-partition, this is, in fact, a $2$-coloring of $K_{R}$. By the definition of $R$, this contains a monochromatic $K_{2, m}$, a contradiction to complete the proof.
\end{proof}

Using Theorem~\ref{Lem:OnlySmall}, the proof of Theorem~\ref{Thm:sis2} and therefore we have the following corollary.

\begin{corollary}\label{Thm:K2m}
Let $R_{2, m} = R(K_{2, m}, K_{2, m})$. For $k \geq 2$ and $m \geq 3$, we have $gr_k (K_3 : K_{2,m}) = R_{2,m} + (k - 2)$.
\end{corollary}

\section{Complete Bipartite Graph $K_{3, m}$}

For complete bipartite graphs $H$ with $s(H) = 3$, we obtain the following results.

\begin{theorem}\label{Thm:K33}
For $k \geq 3$, we have
$$
gr_{k}(K_{3} : K_{3, 3}) = 2k + 14.
$$
\end{theorem}

More generally, for $K_{3, m}$, we have the following result.

\begin{theorem}\label{Thm:K3m}
For $k \geq 3$ and $m \geq 3$, we get
\beqs
R(K_{3, m}, K_{3, m}) + 2(k - 2) & \leq & gr_{k}(K_{3} : K_{3, m})\\
~ & \leq & \max \{(6m - 2), R(K_{3, m}, K_{3, m})\} + 2(k - 2).
\eeqs
\end{theorem}

Although Theorem~\ref{Thm:K33} is actually a corollary of Theorem~\ref{Thm:K3m}, we include the proof of Theorem~\ref{Thm:K33} since it deals with explicit values.

Since $2^{(3m - 1)/(3 + m)} \leq R(K_{3, m}, K_{3, m}) \leq 8m - 2$ from \cite{MP13} and \cite{LM02}, the sharpness of this result in general depends on the Ramsey number. The only small value of $m \geq 3$ for which $R(K_{3, m}, K_{3, m})$ is known is when $m = 3$ (see Theorem~\ref{Thm:RK33}). Since $R(K_{3, 3}, K_{3, 3}) = 18 \geq 6m - 2$, the bounds in Theorem~\ref{Thm:K3m} are equal, establishing the conclusion of Theorem~\ref{Thm:K33} as noted above. Otherwise, the general relationship between $R(K_{3, m}, K_{3, m})$ and $6m - 2$ is not clear.

\begin{proof}[Proof of Theorem~\ref{Thm:K33}]
%For the lower bound, we use the fact that $R(K_{3, 3}, K_{3, 3}) = 18$. Let $H$ be a $2$-coloring of $K_{17}$ containing no monochromatic $K_{3, 3}$. For each additional color $i$ for $3 \leq i \leq k$, we add $2$ vertices $v_{i}$ and $v_{i}'$ with all incident edges in color $i$. This graph certainly contains no rainbow triangle and no monochromatic $K_{3, 3}$ and has order $17 + 2(k - 2) = 2k + 13$.
The lower bound follows from Lemma~\ref{Lem:BipLowBnd}. For the upper bound, suppose $G$ is a G-coloring of $K_{n}$ using at most $k$ colors where $n = 2k + 14$ and suppose $G$ contains no monochromatic $K_{3, 3}$. 
By Theorem~\ref{Lem:OnlySmall}, we may consider a $3$-coloring $G'$ of $K_{18}$ with a G-partition in which all parts have order at most $2$. In particular, since $|G'| = 18$, this means that there are at least $9$ parts in the G-partition.

Let $R$ be the reduced graph of the G-partition of $G'$. Since $R(C_{4}, C_{4}) = 6$, if there at least $5$ parts of the $G$-partition each containing at least $2$ vertices, then the reduced graph corresponding to these parts (along with an additional part if we have only $5$ of order $2$) contains a monochromatic $C_{4}$. Such a subgraph corresponds to a monochromatic complete bipartite subgraph with at least $3$ vertices in each part, a contradiction. We may therefore assume that there are at most $4$ parts of this G-partition which have order $2$ while all the rest have order $1$. On the other hand, since $|G'| = 18$, if the G-partition had only parts of order $1$, then $G'$ would simply be a $2$-coloring and the result would follow from $R(K_{3, 3}, K_{3, 3}) = 18$. We may therefore assume that the number of parts in the partition $t$ satisfies $14 \leq t \leq 17$ and there are between $1$ and $4$ parts of order $2$.

Let $A$ be a part of order $2$. There are at least $16$ vertices remaining in $G'\setminus A$ so at least $8$ of them must have a single color on all edges to $A$, say blue. Let $B$ be a set of $8$ vertices with all blue edges to $A$ (chosen so that if $B$ contains a vertex of a part of the G-partition, then $B$ contains all vertices of that part). Let $v_{1}, v_{2}, v_{3}$ and $v_{4}$ be four vertices in $G' \setminus (A \cup B)$.

To avoid creating a blue $K_{3, 3}$, each vertex $v_{i}$ shares at most $2$ blue neighbors with $A$, so each vertex $v_{i}$ has at least $6$ red edges to $B$. Every pair of vertices $v_{i}, v_{j}$ will then share at least $4$ red neighbors in $B$ and every triple of these vertices must share at least $2$ common red neighbors in $B$. Certainly if a triple shares at least $3$ red neighbors, this would be a red $K_{3, 3}$ so this means that, in particular, $v_{1}, v_{2}, v_{3}$ must share exactly $2$ common red neighbors in $B$ and the red edges from these three vertices to $B$ are strictly prescribed. More specifically, if $b_{1}, \dots, b_{8}$ are the vertices of $B$, we may assume that $v_{1}$ has green edges to $b_{1}, b_{2}$, $v_{2}$ has green edges to $b_{3}, b_{4}$, and $v_{3}$ has green edges to $b_{5}, b_{6}$. Then regardless of the choice of the $6$ red edges from $v_{4}$ to $B$, there are three vertices $v_{i}, v_{j}, v_{\ell}$ which share at least $3$ common red neighbors in $B$, producing a red $K_{3, 3}$ for a contradiction.
\end{proof}

\begin{proof}[Proof of Theorem~\ref{Thm:K3m}]

The lower bound follows from Lemma~\ref{Lem:BipLowBnd}. For the upper bound, suppose $G$ is a G-coloring of $K_{n}$ using at most $k$ colors where 
$$
n = \max \{(6m - 2), R(K_{3, m}, K_{3, m})\} + 2(k - 2)
$$
and suppose $G$ contains no monochromatic $K_{3, m}$. By Theorem~\ref{Lem:OnlySmall}, we may consider a $3$-coloring $G'$ of $K_{n - 2(k - 2)}$ with a G-partition with colors red and blue in which all parts have order at most $2$. In particular, since $|G'| \geq 6m - 2$, this means that there are at least $3m - 1$ parts in this G-partition.

Since $|G'| \geq R(K_{3, m}, K_{3, m})$, there must exist at least one part $X$ of this G-partition of order $2$. First suppose there are at most $2m-1$ vertices in $G' \setminus X$ with some color, say red, on edges to $X$. This means that there are at least
$$
n - 2(k - 2) - 2 - (2m - 1) \geq 4m - 3
$$
vertices in $G' \setminus X$ with all blue edges to $X$. Let $A$ be a set of $4m - 3$ of these vertices. Let $v_{1}, v_{2}, v_{3}$ be three of the vertices in $G' \setminus (X \cup A)$. In order to avoid creating a blue $K_{3, m}$, $v_{i}$ can have at most $m - 1$ blue edges to $A$ so all remaining edges must be red. With $|A| = 4m - 3$, there must be a set of at least $m$ vertices in $A$ with all red edges to $v_{1}, v_{2}, v_{3}$, creating a red $K_{3, m}$. This means that there is no color $c$ (among red and blue) for which at most $2m - 1$ vertices in $G' \setminus X$ have color $c$ on edges to $X$.

Let $A$ (and $B$) be the set of vertices in $G \setminus X$ with blue (respectively red) edges to $X$. Note that we have $|A|, |B| \geq 2m$. To avoid creating a red $K_{3, m}$, each vertex in $A$ must have at most $m - 1$ red edges to $B$. Symmetrically, each vertex in $B$ must have at most $m - 1$ blue edges to $A$. This means that there are a total of at most $(|A| + |B|)(m - 1)$ edges between $A$ and $B$ but since the graph is complete, we know there are exactly $|A|\cdot |B|$ edges between $A$ and $B$. Since $|A|, |B| \geq 2m$, we get
$$
(|A| + |B|)(m - 1) < |A| \cdot |B|,
$$
a contradiction, completing the proof.
\end{proof}

\section{General Complete Bipartite Graphs}

For large complete bipartite graphs, the following was recently shown.

\begin{theorem}[\cite{CLP17}]
For fixed integers $k \geq 2$ and $m \geq 1$, if $\ell \rightarrow \infty$, then
$$
(1 - o(1))2^{m}n \leq gr_{k}(K_{3} : K_{\ell, m}) \leq (2^{m} + 2^{m/2 + 1} + k)n + 4m^{3}.
$$
\end{theorem}

We obtain the following bounds on the Gallai-Ramsey numbers for all complete bipartite graphs.

\begin{theorem}\label{Thm:GenBound}
Given positive integers $\ell, m$ where $\ell \leq m$, let $H = K_{\ell, m}$ and let $R = R(H, H)$. Then
$$
R + (\ell - 1)(k - 2) \leq gr_{k}(K_{3} : H) \leq (R + k - 3)(\ell - 1) + 1.
$$
\end{theorem}

\begin{proof}
The lower bound follows from Lemma~\ref{Lem:BipLowBnd}. For the upper bound, suppose $G$ is a G-coloring of $K_{n}$ using at most $k$ colors where $n = (R - 1)(m - 1) + (\ell - 1)(k - 2)$ and suppose $G$ contains no monochromatic $K_{\ell, m}$. 

By Theorem~\ref{Lem:OnlySmall}, we may consider a $3$-coloring $G'$ of $K_{(R - 1)(\ell - 1) + 1}$ with a G-partition in which all parts have order at most $\ell - 1$. By the definition of $R$, there are at most $R - 1$ parts so with each part having order at most $\ell - 1$, there can be at most $(R - 1)(\ell - 1)$ vertices in $G'$, a contradiction.
\end{proof}

\section{Linear Forests}

% Maybe move this to a section on tP_3 if we end up including that.
It is worth noting that the Gallai-Ramsey number for matchings is exactly the same as the Ramsey number for matchings (proven in \cite{CL75}) since the sharpness example for the Ramsey number contains no rainbow triangle.

\begin{corollary}\label{Cor:Matching}
For positive integers $k, n_{1}, n_{2}, \dots, n_{k}$ with $n_{1} = \max \{n_{i}\}$, we have
\beqs
gr_{k}(K_{3} : n_{1}P_{2}, n_{2}P_{2}, \dots, n_{k}P_{2}) & = & R(n_{1}P_{2}, n_{2}P_{2}, \dots, n_{k}P_{2})\\
~ & = & n_{1} + 1 + \sum_{i = 1}^{k} (n_{i} - 1).
\eeqs
\end{corollary}

For copies of $P_{3}$, the situation is not quite as clear. The $2$-color Ramsey number was solved (in a more general form) in \cite{FS76} with the following result.

\begin{theorem}[\cite{FS76}]\label{Thm:RtP3}
For positive integers $n_{1} \geq n_{2}$, we have
$$
R(n_{1}P_{3}, n_{2}P_{3}) = 3n_{1} + n_{2} - 1.
$$
\end{theorem}

We conjecture that this result extends to more general G-colorings in the following sense.

\begin{conjecture}
For positive integers $k, n_{1}, n_{2}, \dots, n_{k}$ with $n_{1} = \max \{n_{i}\}$, we have
$$
gr_{k}(K_{3} : n_{1}P_{3}, n_{2}P_{3}, \dots, n_{k}P_{3}) = 2n_{1} + 1 + \sum_{i = 1}^{k} (n_{i} - 1).
$$
\end{conjecture}

As a partial result, we prove the following.

\begin{theorem}
For positive integers $k, n_{1}, n_{2}, \dots, n_{k}$ with $n_{1} = \max \{n_{i}\}$, we have
$$
2n_{1} + 1 + \sum_{i = 1}^{k} (n_{i} - 1) \leq gr_{k}(K_{3} : n_{1}P_{3}, n_{2}P_{3}, \dots, n_{k}P_{3})
$$
$$
\leq 4(n_{1} - 1) + \frac{9n_{1} - 3}{2} \ln \left( \frac{3n_{1}}{2} - 1\right) + 1 + (n_{1} - 1)(k - 2).
$$
Moreover, when $n_{1} = 2$, we have
$$
gr_{k}(K_{3} : n_{1}P_{3}, n_{2}P_{3}, \dots, n_{k}P_{3}) = 2n_{1} + 1 + \sum_{i = 1}^{k} (n_{i} - 1).
$$
\end{theorem}

\begin{proof}
If $k = 1$, the result is immediate and if $k = 2$, the result follows from Theorem~\ref{Thm:RtP3}, so suppose $k \geq 3$. For the lower bound, we follow the proof of Lemma~\ref{Lem:BipLowBnd}. Let $G_{1}$ be a copy of $K_{3n_{1} - 1}$ colored entirely with color $1$. Given $G_{i}$, construct $G_{i + 1}$ by adding $n_{i} - 1$ vertices to $G_{i}$ with all edges incident to the new vertices having color $i + 1$. The resulting coloring $G_{k}$ is a coloring of $K_{n}$ where $n =  2n_{1} + \sum_{i = 1}^{k} (n_{i} - 1)$ which contains no rainbow triangle and no monochromatic copy of $n_{i}P_{3}$ in color $i$.

For the upper bound, let $G$ be a G-coloring of $K_{n'}$ where 
$$
n' \geq 4(n_{1} - 1) + \frac{9n_{1} - 3}{2} \ln \left( \frac{3n_{1}}{2} - 1\right) + 1 + (n_{1} - 1)(k - 2).
$$
By Theorem~\ref{Thm:G-Part}, there is a G-partition of $V(G)$. Choosing one vertex from each part of this partition produces a $2$-colored complete graph as a subgraph of the original graph. Supposing that colors $red$ and $blue$ are the two colors appearing in the G-partition with $n_{red} \geq n_{blue}$, this means that by Theorem~\ref{Thm:RtP3}, there are at most $3n_{red} + n_{blue} - 2$ parts. Conversely, by Theorem~\ref{Lem:OnlySmall}, we may consider a $3$-colored $K_{n}$ where 
$$
n \geq 4(n_{1} - 1) + \frac{9n_{1} - 3}{2} \ln \left( \frac{3n_{1}}{2} - 1\right) + 1
$$
with a G-partition in which each part has order at most $n_{red} - 1$.

At this point, we note that if $n_{1} = 2$, each part has order $1$, meaning that the graph is simply a $2$-coloring and the sharp result follows from Theorem~\ref{Thm:RtP3}. It is also worthwhile to observe that we have already shown that $n$ must be at most $(3n_{red} + n_{blue} - 2)(n_{red} - 1)$ so
$$
gr_{k}(K_{3} : n_{1}P_{3}, n_{2}P_{3}, \dots, n_{k}P_{3}) < 4n_{1}^{2} + (n_{1} - 1)(k - 2),
$$
but our goal is closer to a linear bound as a function of $n_{1}$.

%By Theorem~\rn{INSERT REFERENCE}, there is a vertex in $G$ with at least $\frac{2n}{5}$ incident edges in a single color, say color $i$. If $\frac{2n}{5} \geq \max \{3(n_{i} - 1), 2\}$, then we may remove this vertex and apply induction on $\sum n_{i}$. Since $k \geq 3$, if there is a vertex with at least $\frac{3n}{5}$ incident edges in a single color, then the desired inequality trivially holds so we may assume that no such vertex exists. \rn{Now what?}

By Corollary~\ref{Cor:Matching}, there is a monochromatic matching within the reduced graph of this G-partition. Given an integer $t \geq 2$, if we restrict our attention to those ``large'' parts of order at least $\frac{3n_{red}}{2t}$, then if there were at least $3t - 1$ such parts, Corollary~\ref{Cor:Matching} would guarantee the existence of a matching on $t$ edges within the reduced graph of these ``large'' parts. With such a matching, it is easy to construct the desired monochromatic $n_{red}P_{3}$ as follows. For each matching edge, say with corresponding parts $A$ and $B$, select $\frac{n_{red}}{2t}$ vertices in $A$ (call this set $A'$) and $\frac{n_{red}}{2t}$ vertices from $B$ (call this set $B'$) to be the central vertices. For each selected vertex in $A$, choose two previously unused vertices of $B$ to construct a copy of $P_{3}$ and similarly for each vertex in $B$, choose two previously unused vertices of $A$ to construct a copy of $P_{3}$. With $t$ edges within the matching, this results in $t \frac{2 n_{red}}{2t} = n_{red}$ copies of $P_{3}$ all in one color, as desired. See Figure~\ref{Fig:Matching} for an example of this construction. Thus, we arrive at the following fact.

\begin{figure}[H]
\begin{center}
\includegraphics{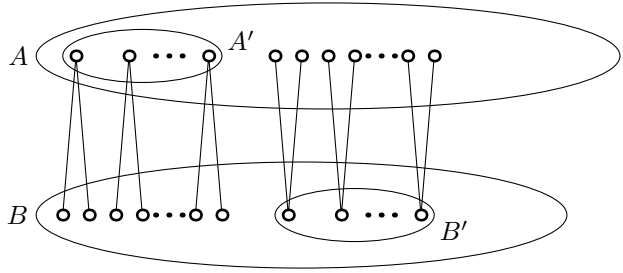}
\caption{Monochromatic copies of $P_{3}$ \label{Fig:Matching}}
\end{center}
\end{figure}

\begin{fact}\label{Fact:LogFact}
For any integer $t \geq 2$, there can be at most $3t - 2$ parts of order at least $\frac{3n_{red}}{2t}$.
\end{fact}

In particular, this means that there can be up to $4$ parts of order $n_{red} - 1$ but the $5^{th}$ part must have order at most $\frac{3n_{red} - 1}{4}$. Similarly, the $8^{th}$ largest part must have order at most $\frac{3n_{red} - 1}{6}$ and so on. This means that
\beqs
n & \leq & 4(n_{red} - 1) + \sum_{t = 3}^{\frac{3n_{red}}{2}} 3\left( \frac{3n_{red} - 1}{2t - 2} \right)\\
~ & \leq & 4(n_{red} - 1) + \frac{9n_{red} - 3}{2} \sum_{t = 3}^{\frac{3n_{red}}{2}} \frac{1}{t - 1} \\
~ & < & 4(n_{red} - 1) + \frac{9n_{red} - 3}{2} \ln \left( \frac{3n_{red}}{2} - 1\right)\\
~ & \leq & 4(n_{1} - 1) + \frac{9n_{1} - 3}{2} \ln \left( \frac{3n_{1}}{2} - 1\right),
\eeqs
contradicting the assumed lower bound on $n$ and completing the proof.
\end{proof}

%\rn{CAN WE GET THE SHARP RESULT FOR THE CASE WHERE $n_{1} \leq 2$?}

It appears as though the general multicolor classical Ramsey number for copies of $P_{3}$ is not known so we conjecture the following.

\begin{conjecture}
For positive integers $k, n_{1}, n_{2}, \dots, n_{k}$ with $n_{1} = \max \{n_{i}\}$, we have
\beqs
R(n_{1}P_{3}, n_{2}P_{3}, \dots, n_{k}P_{3}) & = & gr_{k}(K_{3} : n_{1}P_{3}, n_{2}P_{3}, \dots, n_{k}P_{3})\\
~ & = & 2n_{1} + 1 + \sum_{i = 1}^{k} (n_{i} - 1).
\eeqs
\end{conjecture}

\section{Conclusion and Further Discussion}\label{Sec:Conj}

Based on our results and, in fact, the entire literature of results on Gallai-Ramsey numbers for a rainbow triangle or monochromatic bipartite graph, we propose the following broad conjecture.

\begin{conjecture}\label{Conj:GR-Bip}
Given any connected bipartite graph $H$ with Ramsey number $R(H, H) = R$, we believe that
$$
gr_{k}(K_{3} : H) = R + (s(H) - 1)(k - 2).
$$
\end{conjecture} 

In particular, this would mean that the lower bound provided by Lemma~\ref{Lem:BipLowBnd} is always sharp.

\bibliography{ref}

\begin{thebibliography}{10}

\bibitem{AI08}
M.~Axenovich and P.~Iverson.
\newblock Edge-colorings avoiding rainbow and monochromatic subgraphs.
\newblock {\em Discrete Math.}, 308(20):4710--4723, 2008.

\bibitem{B83}
S.~A. Burr.
\newblock Diagonal {R}amsey numbers for small graphs.
\newblock {\em J. Graph Theory}, 7(1):57--69, 1983.

\bibitem{CE97}
K.~Cameron and J.~Edmonds.
\newblock Lambda composition.
\newblock {\em J. Graph Theory}, 26(1):9--16, 1997.

\bibitem{CLP17}
M.~Chen, Y.~Li, and C.~Pei.
\newblock Gallai-{R}amsey numbers for odd cycles and complete bipartite graphs.
\newblock {\em To appear}.

\bibitem{CG83}
F.~R.~K. Chung and R.~L. Graham.
\newblock Edge-colored complete graphs with precisely colored subgraphs.
\newblock {\em Combinatorica}, 3(3-4):315--324, 1983.

\bibitem{CL75}
E.~J. Cockayne and P.~J. Lorimer.
\newblock The {R}amsey number for stripes.
\newblock {\em J. Austral. Math. Soc.}, 19:252--256, 1975.

\bibitem{EHM91}
G.~Exoo, H.~Harborth, and I.~Mengersen.
\newblock On {R}amsey numbers of {$K_{2,n}$}.
\newblock In {\em Graph theory, combinatorics, algorithms, and applications
  ({S}an {F}rancisco, {CA}, 1989)}, pages 207--211. SIAM, Philadelphia, PA,
  1991.

\bibitem{FS76}
R.~J. Faudree and R.~H. Schelp.
\newblock Ramsey numbers for all linear forests.
\newblock {\em Discrete Math.}, 16(2):149--155, 1976.

\bibitem{FGP15}
J.~Fox, A.~Grinshpun, and J.~Pach.
\newblock The {E}rd{\H o}s-{H}ajnal conjecture for rainbow triangles.
\newblock {\em J. Combin. Theory Ser. B}, 111:75--125, 2015.

\bibitem{FMO10}
S.~Fujita, C.~Magnant, and K.~Ozeki.
\newblock Rainbow generalizations of {R}amsey theory: a survey.
\newblock {\em Graphs and Combin.}, 1:1--30, 2010.

\bibitem{FMO14}
S.~Fujita, C.~Magnant, and K.~Ozeki.
\newblock Rainbow generalizations of {R}amsey theory: a dynamic survey.
\newblock {\em Theo. Appl. Graphs}, 0(1):Article 1, 2014.

\bibitem{G67}
T.~Gallai.
\newblock Transitiv orientierbare {G}raphen.
\newblock {\em Acta Math. Acad. Sci. Hungar}, 18:25--66, 1967.

\bibitem{GSSS10}
A.~Gy{\'a}rf{\'a}s, G.~S{\'a}rk{\"o}zy, A.~Seb{\H o}, and S.~Selkow.
\newblock Ramsey-type results for gallai colorings.
\newblock {\em J. Graph Theory}, 64(3):233--243, 2010.

\bibitem{GS04}
A.~Gy{\'a}rf{\'a}s and G.~Simonyi.
\newblock Edge colorings of complete graphs without tricolored triangles.
\newblock {\em J. Graph Theory}, 46(3):211--216, 2004.

\bibitem{HM91}
H.~Harborth and I.~Mengersen.
\newblock The {R}amsey number of {$K_{3,3}$}.
\newblock In {\em Graph theory, combinatorics, and applications. {V}ol.\ 2
  ({K}alamazoo, {MI}, 1988)}, Wiley-Intersci. Publ., pages 639--644. Wiley, New
  York, 1991.

\bibitem{LMSSS17}
H.~Liu, C.~Magnant, A.~Saito, I.~Schiermeyer, and Y.~Shi.
\newblock Gallai-{R}amsey number for ${K}_{4}$.
\newblock {\em Submitted}.

\bibitem{LM02}
R.~Lortz and I.~Mengersen.
\newblock Bounds on {R}amsey numbers of certain complete bipartite graphs.
\newblock {\em Results Math.}, 41(1-2):140--149, 2002.

\bibitem{MP13}
T.~K. Mishra and S.~P. Pal.
\newblock Lower bounds for {R}amsey numbers for complete bipartite and
  3-uniform tripartite subgraphs.
\newblock {\em J. Graph Algorithms Appl.}, 17(6):671--688, 2013.

\bibitem{MR1670625}
S.~P. Radziszowski.
\newblock Small {R}amsey numbers.
\newblock {\em Electron. J. Combin.}, 1:Dynamic Survey 1, 30 pp. (electronic),
  1994.

\bibitem{R04}
V.~Rosta.
\newblock Ramsey theory applications.
\newblock {\em Electron. J. Combin.}, 11(1):Research Paper 89, 48, 2004.

\end{thebibliography}
\bibliographystyle{plain}

\end{document}